\documentclass[bjps,preprint]{imsart}

\RequirePackage[OT1]{fontenc}
\usepackage{amsthm,amsmath,natbib,amsfonts,graphicx}
\RequirePackage{hyperref}

\arxiv{arXiv:1308.6626}

\startlocaldefs
\numberwithin{equation}{section}
\theoremstyle{plain}
\newtheorem{thm}{Theorem}[section]

\newtheorem*{proposition*}{Proposition}
\newcommand{\cG}{\mathcal G}
\endlocaldefs

\begin{document}

\begin{frontmatter}
\title{Searching for the core variables in principal components analysis}
\runtitle{Searching for the core variables in principal components analysis}

\begin{aug}
\author{\fnms{Yanina} \snm{Gimenez}\thanksref{a}\ead[label=e1]{yanugimenez@gmail.com}},
\and
\author{\fnms{Guido} \snm{Giussani}\thanksref{a}%
\ead[label=e2]{ggiussani@udesa.edu.ar}}

\runauthor{Gimenez and Giussani}

\affiliation[a]{Universidad de San Andr\'es and CONICET}

\address{Universidad de San Andr\'es, Vito Dumas 284, 1644 Victoria, Buenos Aires, Argentina and CONICET\\
\printead{e1,e2}}

\end{aug}

\begin{abstract}
In this article, we introduce a procedure for selecting variables
in principal components analysis. It is developed to
identify a small subset of the original variables that best
explain the principal components through nonparametric
relationships. There are usually some noisy uninformative
variables in a dataset, and some variables that are strongly
related to one another because of their general dependence.
The procedure is designed to be used following the satisfactory
initial principal components analysis with all variables,
and its aim is to help to interpret the underlying structures. We
analyze the asymptotic behavior of the method and provide some
examples.
\end{abstract}



\end{frontmatter}

\section{Introduction}\label{AC-Introd}

Principal components analysis (PCA) is the best known dimensional
reduction procedure for multivariate data. An important drawback of PCA is
that it sometimes provides poor quality interpretation of data for
practical problems, because the final output is a
linear combination of the original variables. The aim of the present
study is to identify a small subset of the original variables in a dataset, whilst retaining
most of the information related to the
first $k$ principal components.

There is a large body of literature that focuses on trying to
interpret principal components.
\citet{J95} has introduced
rotation techniques.
\citet{V00} proposed to restrict the
value of the loadings for PCA to a small set of allowable integers
such as \{-1, 0, 1\}.
\citet{M84} presented a different
strategy that aims to select a subset of the original variables with
a similar criterion to PCA.

A few years later a whole literature of variable selection appeared,
inspired by the LASSO (least absolute shrinkage and selection
operator). This technique was introduced by
\citet{lasso}. The way the LASSO works was described thus: `` It
shrinks some coefficients and sets others to 0, and hence tries to
retain the good features of both subset selection and ridge
regression... The `LASSO' minimizes the residual sum of squares
subject to the sum of the absolute value of the coefficients being
less than a constant. Because of the nature of this constraint it
tends to produce some coefficients that are exactly 0 and hence
gives interpretable models."

\citet{JTU03} proposed SCoTLASS that imposes a bound
on the sum of the absolute values of the loadings in the component,
using a similar idea to the one that LASSO used in regression.
\citet{ZHT06} presented SPCA (sparse PCA) that extends the
``elastic net" of
\citet{ZH05} that was a generalization of the LASSO. As they
mentioned ``SPCA is built on the fact that PCA can be written as a
regression-type optimization problem, with quadratic penalty; the
LASSO penalty (via the elastic net) can then be directly integrated
into the regression criterion, leading to a modified PCA with sparse
loadings."

\citet{clustering} studied the application of
sparse PCA to clustering and problems of feature selection. Sparse
PCA seeks sparse factors, or linear combinations of variables in the
dataset, to explain as much variance in the data as possible while
limiting the number of nonzero coefficients as far as possible. The
authors applied their results to some classic biological clustering
and feature selection problems.

Recently,
\citet{lasso for genes} introduced
the notion of Lassoed principal components for identifying
differentially-expressed genes, and considered the problem of
testing the significance of features in high dimensional data.

Our approach is rather different and is designed to be used after
satisfactory PCA has been achieved rather than, as in other
methods, to produce principal components with particular
characteristics (e.g., some coefficients that are zero) so that
only interpretable principal components are produced. We first
perform classical PCA and then look for a small subset of the
original variables that contain almost all the relevant
information to explain the principal components.
However, our method can also be used after
performing any sparse PCA method as those described previously.

To develop our method, we borrowed some ideas for selecting
variables from \citet{selection}, who introduced two
procedures for selecting variables in cluster analysis, and
classification rules. Both of these procedures are based on the idea
of ``blinding" unnecessary variables. To cancel the effect of a
variable, they substituted all its values with the marginal mean (in
the first procedure) or with the conditional mean (in the second).
The marginal mean approach was mainly intended to identify ``noisy"
uninformative variables, but the conditional mean approach could
also deal with dependence. We adapted the idea behind the second
procedure to PCA.

In Section 2, we introduce the main definitions, a population
version of our proposed method, and the empirical version; we also
present our main results. In Section 3, a simulation study is
conducted, where the results are compared with other well-known
PCA variable selection procedures. Finally, in Section 4, we study a
real data example. Proofs are given in the Appendix.

\section{Our method: Definitions and properties}

We begin by defining our notation and stating the problem in terms
of the underlying distribution of the random vector $\mathbf{X}$.
Then we give our estimates based on the sample data, via the
empirical distribution.

\subsection{Population version}

We define $\mathbf{X}  \in \mathbb{R}^p$ as a random vector with distribution $\mathbb{P}$. The coordinates of the vector
$\mathbf{X}$ are defined as $X[i], \ \ i=1, \ldots, p$. The covariance matrix of $\mathbf{X}$ is denoted $\Sigma$.

For a given random vector $\mathbf{\tilde{X}}$, we say that it satisfy the \bf Assumption H1 \rm if:

\begin{itemize}
\item [i)] $\mathbb{E}(\Vert \mathbf{\tilde{X}} \Vert^2) < \infty$;
\item [ii)] The covariance matrix of $\mathbf{\tilde{X}}$ is positive definite.
\item [iii)] All the eigenvalues of the covariance matrix are different.
\end{itemize}

Throughout the manuscript, $\mathbf{X}$ fulfills \bf Assumption H1\rm.

As is well known, the first principal component associated with the
vector $\mathbf{X}$ is defined as
\begin{eqnarray*}
\alpha^1(\mathbb{P}):= \alpha^1 &=& \underset{\left\Vert\alpha\right\Vert=1}{arg \ max} \
Var\left(\mathbf{\alpha'}\mathbf{X}\right)\\&=&
\underset{\left\Vert\mathbf{\alpha}\right\Vert=1}{arg \ max} \ \mathbf{\alpha}' \Sigma
\mathbf{\alpha},
\end{eqnarray*}
and the next principal components are defined as
\begin{eqnarray*}
\alpha^k(\mathbb{P}):= \alpha^k &=& \underset{\left\Vert\mathbf{\alpha}\right\Vert=1, \ \mathbf{\alpha} \bot [\alpha^1, \ldots,\alpha^{k-1}]}{arg \ max} \
Var\left(\mathbf{\alpha}'\mathbf{X}\right)\\&=&
\underset{\left\Vert\mathbf{\alpha}\right\Vert=1, \ \mathbf{\alpha} \bot [\alpha^1, \ldots,\alpha^{k-1}]}{arg \ max} \ \mathbf{\alpha}' \Sigma \mathbf{\alpha} \
\ \forall \ 2\leq k\leq p,
\end{eqnarray*}
where $[\alpha^1, \ldots,\alpha^{k-1}]$ is the subspace generated by the vectors $\{\alpha^1, \ldots,\alpha^{k-1}\}$.

From the spectral theorem, it follows that, if $\lambda^1 >
\lambda^2
> \ldots
> \lambda^p$ are the $\Sigma$ eigenvalues, the
solutions to the PCA are
the corresponding
eigenvectors, $\alpha^k, \ \ k=1, \ldots, p$.

\

Given a subset of indices $I \subset \{1, \ldots, p\}$ with
cardinality $d\leq p$, we define $\mathbf X[I]$ as the subset
 of random variables $\{X[i], i \in I \}$. With a slight
 abuse of notation, if $I=\{i_1 < \ldots < i_d\}$, we
 can also denote $\mathbf X[I]$ to the vector $(X[i_1], \ldots,X[i_d])$,
 and define the vector $\mathbf{Y}^{I}:=(Y^{I}[1],\ldots,Y^{I}[p])$, where

\begin{equation*}
\label{Y} Y^{I}[i]= \left\{
\begin{array}{ccc}
X[i] &&\text{ if }i \in I
\\  \mathbb{E}(X[i] \vert \mathbf X[I]) &&\text{ if } i \notin I.
\end{array}
\right.
\end{equation*}
$\mathbf{Y}^{I} \in \mathbb{R}^p$ depends only on the
$\{X[i], i \in I \}$ variables and then the  principal components associated with the
vector $\mathbf{Y}^{I}$ depend only on the variables in $\mathbf X[I]$. The distribution of
$\mathbf{Y}^{I}$ is denoted $\mathbb{P}_{\mathbf{Y}^{I}}$, the covariance matrix is $\Sigma_{\mathbf{Y}^{I}}$ and $g^i(z) = \mathbb E(X[i] \vert \mathbf X[I] = z)$ for $i\notin I$ is the regression function.

%\textcolor{magenta}{\bf Remark: \rm In what follows we need that the covariance matrix of $\mathbf{Y}^{I}$ is non singular.
%The covariance matrix of $\mathbf{Y}^{I}$ is singular if and only if there exists a non stochastic number $c$ and a non stochastic vector $b\neq 0$ such that $P(b'\mathbf{Y}^{I}=c)=1$. From \bf Assumption H1 \rm we know that there does not exist a linear combination of the variables $i\in I$ that equals a constant with probability 1. But, for example, it can happpens that the regression function $g^i(z)$ is a linear combination of the variables in $I$. This is not a common situation. In what follows, we are going to assume that this does not happen.
%}

In what follows we assume that $\mathbf{Y}^{I}$ fulfills the \bf Assumption H1\rm.

\bf Remark: \rm This assumptions will not hold if there is a variable $X[i]$ such that $g^i(z)$ is a linear combination of the variables in $I$. For instance, if $\mathbf{X}$ is gaussian or $X[i]$ is independent of the variables in $I$. In practice we will see how this problem can be tackled in the example 2 of section 3.

We are looking for a subset $I$ of small cardinality that minimizes the distance between the original principal components and the
principal components that are function of the variables in the subset $I$. This can be done following two different approaches.

\

\bf Local approach \rm

We define the objective function $h^1(I)$ as
\begin{equation}
\label{functionh1}
h^1(I, \mathbb{P} , \mathbb{P}_{\mathbf{Y}^{I}}):= h^1(I,\mathbb{P}):= h^1(I)=
\Vert \alpha^1(\mathbb{P}) - \alpha^1(\mathbb{P}_{\mathbf{Y}^{I}})\Vert^2,
\end{equation}
which measures the squared distance between the first original principal component and the
first principal component that is a function of the variables  in the subset $I$.

Given a fixed integer $d$, $1\leq d <<p$, $\mathcal{I}_d$ is
the family of all subsets of $\{1, \ldots, p\}$ with cardinality
$d$ and $\mathcal{I}_{1,0} \subset \mathcal{I}_d$ is the family of subsets in which
the minimum $h^1(I)$ is attained for $I \in \mathcal{I}_d$,
i.e.,
\begin{equation*}
\label{objetivo1} \mathcal{I}_{1,0} = argmin_{I \in \mathcal{I}_d}
h^1(I),
\end{equation*}
or, equivalently,
\begin{equation}
\label{henobjetivo1}
h^1(I_1)=min_{I \in \mathcal{I}_d} h^1(I)  \text{ for all } I_1 \in \mathcal{I}_{1,0}.
\end{equation}

Analogously, we define
\begin{equation}
\label{functionhk}
h^k(I,\mathbb{P}, \mathbb{P}_{\mathbf{Y}^{I}}):=h^k(I,\mathbb{P}):= h^k(I)= \Vert \alpha^k(\mathbb{P}) - \alpha^k(\mathbb{P}_{\mathbf{Y}^{I}})\Vert^2,
\end{equation} and
\begin{equation*}
\label{henobjetivoj} \mathcal{I}_{k,0} = argmin_{I \in
\mathcal{I}_d} h^k(I),
\end{equation*}
 for $k=2, \ldots, p$.

For $k=1, \ldots, p$ (\ref{functionhk}) measures the squared
distance between the $k$-th original principal component and the
$k$-th principal component that is a function of the variables in
the subset $I$.

Principal components may, in some cases, be rather difficult to
interpret, which is an important issue in practice.
If we can find a small cardinal subset $I$, for which the objective
function (\ref{functionhk}) is small enough, then the $k$-th principal component will be well explained by a subset of the original variables.

\

\bf Global approach\rm

In this case, the objective function is
\begin{equation}
\label{objetivo}
 h(I):=\sum_{k=1}^q p_kh^k(I),
\end{equation}
with $p_k\geq 0,\ \sum_{k=1}^q p_k=1,\ 2\leq q<p$,\\
and we define
\begin{equation*}
\label{Iqobjetivo} \tilde {\mathcal{I}}_{q,0}= argmin_{I \in \mathcal{I}_d} h(I).
\end{equation*}

This time the objective function (\ref{objetivo}) deals with
finding a unique subset $I$ to explain the first $q$ principal
components at once. If we think that the $q$ components are equally
important then we choose $p_k= \frac{1}{q}$. Otherwise, if some
components are more important than others then we can put
different weights. As a default we suggest choosing weights
proportional to the variance that each component is explaining,
i.e. $p_k=\lambda^k /\sum_{k=1}^{q}\lambda^k$.

\

On the Local approach, we consider a different subset for each principal component, using the objective function (\ref{functionhk}) for each $k$.
On the Global approach, we seek a unique subset for all the first $q$ principal components using (\ref{objetivo}). In practice, we can choose $q$, so that the first $q$ principal components explain a high percent of the total variance and consider a subset for this principal components.

These subsets tell us which are the original variables that ``best explain" the first $q$ principal components. In what follows we refer to this method as the \it Blinding Procedure \rm\bf(BP)\rm.

\

An important issue is how to choose $d$, the cardinality of the
set $I$. On the one hand we have that the objective function
(\ref{functionhk}) decreases when $d$ increases for every $k$. On the
other hand we look for a subset $I$ with small cardinality, so we have a
constant tug of war. Since the components have unitary norms there
is a direct relationship between (\ref{functionhk}) and the angle between the two vectors. It is clear that the smaller the angle is the closer the components are.

Hence, for the $k$-th component we propose to fix an angle, $\gamma$, and choose $d$ as the smallest value for which $\mathcal{I}_{k,0}(d)$ makes the angle
between the blinding and the original component smaller than
$\gamma$.

If we want ``to explain'' the first $q$ principal components at once,  we propose to fix an angle, $\gamma$, and choose $d$ as the smallest value for which $\tilde {\mathcal{I}}_{q,0}(d)$ makes the largest
of the $q$ angles smaller than $\gamma$. In both cases, the angle $\gamma$ must be chosen by the user. As
default, we propose to fix an angle not larger than 25 degrees.

\subsection{Empirical version. Consistent estimates of the optimal
subset}

We aimed to consistently estimate the sets $I_1, \ldots, I_q$ from
a sample $\textbf{X}_1, \ldots, \mathbf{X}_n$ of i.i.d. random vectors with
distribution $\mathbb{P}$. \rm

\

Given a subset $I \in \mathcal{I}_d$, the first step is to
build a sample $ \mathbf{Y}_1^{I}, \ldots, \mathbf{Y}_n^{I}$ of random vectors in
$\mathbb R^p$, which only depends on $ \mathbf{X}[I]$, using nonparametric
estimates of the conditional expectation (the regression function). Below, we assume that

\bf Assumption H2 \rm For all $i \notin I$, $g^i_n(z)$ is a strongly consistent estimate
of $g^i(z)$ uniformly in $z$. Conditions under which \bf H2 \rm holds can be found in \citet{Han}.

\

First, we define the empirical version of the \it blinded \rm observations. As an example, we consider
the nearest neighbours estimator. We therefore set an integer value $r$ (the number of
nearest neighbours that we are going to use) with respect to an
appropriate metric (typically Euclidean or Mahalanobis distance),
only considering the coordinates on $I$.

More precisely, for each $j \in \{1, \ldots , n\}$, we find the
set of indices $C_j=: C_j(X_j[I])$ of the $r$-nearest neighbours of
$\mathbf{X}_j[I]$ among $\{\mathbf{X}_1[I], \ldots,\mathbf{X}_n[I]\}$, where
$\mathbf{X}_j[I]=\{X_j[i],i\in I\}$.

Next, we define the random vectors $\mathbf{Y}_j^{I} , 1 \leq j \leq
n$, verifying:
\begin{eqnarray}\label{def new sample}
Y_j^{I} [i] = \left\{
\begin{array}{ccc}
 X_j [i] &&\text{ if }i \in I
\\ \frac{1}{r}\sum_{m \in C_j} X_{m}[i] &&\text{ otherwise,}
\end{array}
\right.
\end{eqnarray}
where $X_j[i]$ stands for the ith-coordinate of the vector
$\mathbf{X}_j$.

This corresponds to use for $i \notin I$ the
nonparametric estimate
%\begin{equation*}
%g^i_n(z[I]) = \frac{1}{r}\sum_{m \in C_j(z[I])} X_{m}[i].
%\end{equation*}
\begin{equation*}
g^i_n(z) = g^i_n(z[I]) = \frac{1}{r}\sum_{m \in C_j(z[I])} X_{m}[i],
\end{equation*}

since $g^i_n(z)$ is a strongly consistent estimate
of $g^i(z)$, where $g^i(z) = \mathbb E(X[i] \vert \mathbf X[I] = z)$. Moreover, observe that the notation $z[I]$ is to emphasize that the function depends on the coordinates indicated by the subset $I$.

$\mathbb{P}_{n}$ stands for the empirical distribution associated with
$\{\mathbf{X}_j , 1 \leq j \leq n\}$ and $\mathbb{P}_{n,\mathbf{Y}^{I}}$ stands for the empirical distribution of
$\{\mathbf{Y}_j^{I} , 1 \leq j \leq n\}$.

In our examples to consistently estimate the optimal number of
nearest neighbours we use the generalized cross validation procedure
proposed by \citet{LG87}, i.e.
\begin{equation*}
\widehat{r}_{opt}(i,I)=\arg \min_r \frac{\frac{1}{n} \sum_{j=1}^n\left( X_j[i]-Y_j[i]\right)^2}{\left(1- \frac{1}{r}\right)^2},\quad i \notin I.
\end{equation*}

When $d$ is large, nonparametric estimators perform
poorly due to the curse of dimensionality. In this case, a
semi-parametric approach should be used like that proposed by
\citet{HS96}. Also a recent proposal by \citet{BFGM13} called
COBRA can be used to avoid the curse of dimensionality. The
consistency result given in the following theorem will still be
valid as long as the semi-parametric estimates verify the
consistency assumptions required for the purely nonparametric
estimates.

We will be mainly interested in
cases where $d$ is small. In our experience a good idea is to start
the search with some genetic algorithm which  provides an initial
solution with a few variables ($d$ small) and then improve  the
result working with $card(I) \leq d$. Also a forward-backward
algorithm as the one proposed in \cite{selection} can be used.

\

Finally, we define, for each $I \in \mathcal{I}_d$, the corresponding empirical versions
\begin{eqnarray}\label{empirical}\hat{\mathbf{\alpha}}^k_n(I):= \alpha^k(\mathbb{P}_{n,\mathbf{Y}^{I}}), \hat h^k_n(I) =
h^k(I, \mathbb{P}_n,\mathbb{P}_{n,\mathbf{Y}^{I}}), \hat I_{kn} = argmin_{I \in \mathcal{I}_d} \hat h^k_n(I)\end{eqnarray}
and

\begin{eqnarray}\label{empiricalOBJ}
\hat h_n(I)=\sum_{k=1}^q p_k \hat h _n^{k}(I), \ \ \hat I_n = argmin_{I \in \mathcal{I}_d} \hat h_n(I)
\end{eqnarray}

respectively.

Let us observe that $\hat{\mathbf{\alpha}}^k_n(I)$ corresponds to the principal component associated with the sample $\{\mathbf{Y}_j^{I} , 1 \leq j \leq n\}$. $\hat h^k_n(I)$ corresponds to the objective function of the k-th principal component, it measures the square distance between the k-th original principal component and the k-th principal component that is a function of the variables in the subset $I$. The function $\hat h^k_n(I)$ determines which are the variables that retain most of the information related to the k-th principal component. The variables we are choosing are the ones that minimize the function $\hat h^k_n(I)$. The subset $\hat I_{kn}$ indicates which this variables are. In case we are looking for a unique subset $I$ to explain the first $q$ principal components, we consider the function $\hat h_n(I)$.\\

A robust version can be obtained using robust principal components
(see for instance \citet{conv robusta}) and replacing
in (\ref{def new sample}) the local mean by local median.

\

\begin{thm} \label{consistencia} Under assumptions H1 for $\mathbf{X}$ and $\mathbf{Y}^{I}$ and H2 we have that
$\hat I_{kn} \in \mathcal{I}_{k,0}$ ultimately for $k=1, \ldots, q$, i.e.,
$\hat{I}_{kn}=I_k$ with $I_k \in \mathcal{I}_{k,0}$ $\forall n> n_0(\omega)$,
with probability one. We also have that $\hat I_{n}\in \mathcal{\tilde{I}}_{q,0}$ ultimately.
\end{thm}

\begin{proof}
The proof is given in the Appendix.
\end{proof}

\section{Some simulated examples}
In this section we consider two simulated experiments to  analyze
the behavior of our method and compare it with other methods
proposed in the literature.

\subsection{Example 1}
To better understand the heart of our procedure we
start with a simple simulation example in dimension four.

There are two ``hidden" factors:
\begin{eqnarray*}
 && V_1\sim N(0,1.25^2),
\\&& V_2\sim N(0,0.55^2),
\end{eqnarray*}
where $V_1$, $V_2$ are independent.

We construct the 4 observable variables as
follows:
\begin{eqnarray} \label{modelito}
 && X_1=V_1 + \epsilon_1,
\\&& X_2=|V_1| + \epsilon_2, \nonumber
\\&& X_3=V_2 + \epsilon_3, \nonumber
\\&& X_4=V_1 V_2 + \epsilon_4, \nonumber
\end{eqnarray}
where $\varepsilon_j$, $1\leq j\leq4$ are i.i.d. $N(0,\sigma^2)$, with $\sigma=0.01$ (Case 1), $\sigma=0.1$ (Case 2) and $\sigma=0.25$ (Case 3).

When we perform the variable selection it is clear that two variables containing the information given by $V_1$ and $V_2$ should be kept. This means that $\{X_1,X_3\}$, $\{X_1, X_4\}$ or
$\{X_3, X_4\}$ would be ``good" choices while  $\{X_1, X_2\}$ is
clearly a ``bad" choice since only retains information of $V_1$. Additionally, $\{X_2, X_3\}$ (respec. $\{X_2, X_4\}$) is not a ``good'' choice since can not recover all the information of $V_1$ (respec. $V_1$ and $V_2$).

In this example, in all the cases, we compare our procedure with other proposals:
algorithms $B2$ and $B4$ (\citet{J02}), a variable selection approach
proposed by \citet{M84} and sparce PCA introduced by \citet{ZHT06}.
To retain $q$ variables, algorithm $B2$ associates one of the original variables to each of the $p-q$ last PCA vectors and deletes those variables,
while $B4$ associates one of the original
variables to each of the first $q$ PCA vectors and retains those variables.

For each case, we perform 500 replicates and
on each of them we generate samples of size 100 from  model
(\ref{modelito}) and compute the covariance matrix.

\begin{itemize}
	\item \bf Case 1:\rm \ The first two principal components explain between
70 and 85 percent of the total variance. For the \bf BP\rm \ 85\% of the times the maximum angle conformed
by the first two principal components is smaller than 25 degrees, hence we keep two variables.\\
The results exhibitted in Table \ref{results dim 4-0.01} on page ~\pageref{results dim 4-0.01} show that 78\% of the times
the \bf BP\rm \ makes a ``good'' choice of variables, $B2$ only
does it 27\% of the times, while the other methods fail in more than 75\% of the replications.

\begin{table}[!h]
\begin{center}
\caption{Proportion of times where each method selects each pair of variables.(Case 1)}
\label{results dim 4-0.01}
\begin{tabular}{|l|l|l|l|l|l|}
\hline
 & \bf BP\rm & B2  &  B4 & McCabe & SPCA\\ \hline
$X_1,X_2$ & $0.22$ & $0.732$ & $0.76$ & $0.782$ & $0.76$ \\ \hline
$X_1,X_3$ & $0.37$ & $0.004$ & $0.002$ & $0.002$ & $0.002$ \\ \hline
$X_1,X_4$ & $0.366$ & $0.264$ & $0.238$ & $0.216$ & $0.238$ \\ \hline
$X_2,X_3$ & $0$ & $0$ & $0$ & $0$ & $0$ \\ \hline
$X_2,X_4$ & $0$ & $0$ & $0$ & $0$ & $0$ \\ \hline
$X_3,X_4$ & $0.044$ & $0$ & $0$ & $0$ & $0$ \\ \hline
\end{tabular}
\end{center}
\end{table}

  \item \bf Case 2:\rm  \ The first two principal components explain between
69 and 85 percent of the total variance. For the \bf BP\rm \ more than 78\% of the times the maximum angle conformed
by the first two principal components is smaller than 25 degrees, hence we keep two variables.\\
 The results exhibitted in Table \ref{results dim 4-0.1} on page ~\pageref{results dim 4-0.1} show that that 75\% of the times
the \bf BP\rm \ makes a ``good'' choice of variables, $B2$ only
does it 29\% of the times, $B4$ and SPCA 27\%, while McCabe fails in more than  75\% of the replications.

\begin{table}[!h]
\begin{center}
\caption{Proportion of times where each method selects each pair of variables.(Case 2)}
\label{results dim 4-0.1}
\begin{tabular}{|l|l|l|l|l|l|}
\hline
 & \bf BP\rm & B2  &  B4 & McCabe & SPCA\\ \hline
$X_1,X_2$ & $0.248$ & $0.714$ & $0.728$ & $0.752$ & $0.728$ \\ \hline
$X_1,X_3$ & $0.354$ & $0.002$ & $0.002$ & $0.002$ & $0.002$ \\ \hline
$X_1,X_4$ & $0.368$ & $0.284$ & $0.27$ & $0.246$ & $0.27$ \\ \hline
$X_2,X_3$ & $0$ & $0$ & $0$ & $0$ & $0$ \\ \hline
$X_2,X_4$ & $0$ & $0$ & $0$ & $0$ & $0$ \\ \hline
$X_3,X_4$ & $0.03$ & $0$ & $0$ & $0$ & $0$ \\ \hline
\end{tabular}
\end{center}
\end{table}

  \item \bf Case 3:\rm \ The first two principal components explain between
67 and 82 percent of the total variance. For the \bf BP\rm, \ 69\% of the times the maximum angle conformed
by the first two principal components is smaller than 25 degrees, hence we keep two variables.\\
 The results exhibitted in Table \ref{results dim 4-0.25} on page ~\pageref{results dim 4-0.25} show that 74\% of the times the \bf BP\rm \ makes a ``good'' choice of variables, $B2$ does it 51\% of the times, $B4$ and SPCA 48\%, while McCabe 46\%.
\begin{table}[!h]
\begin{center}
\caption{Proportion of times where each method selects each pair of variables.(Case 3)}
\label{results dim 4-0.25}
\begin{tabular}{|l|l|l|l|l|l|}
\hline
 & \bf BP\rm & B2  &  B4 & McCabe & SPCA\\ \hline
$X_1,X_2$ & $0.264$ & $0.488$ & $0.516$ & $0.536$ & $0.516$ \\ \hline
$X_1,X_3$ & $0.192$ & $0.002$ & $0$ & $0$ & $0$ \\ \hline
$X_1,X_4$ & $0.54$ & $0.51$ & $0.484$ & $0.464$ & $0.484$ \\ \hline
$X_2,X_3$ & $0$ & $0$ & $0$ & $0$ & $0$ \\ \hline
$X_2,X_4$ & $0$ & $0$ & $0$ & $0$ & $0$ \\ \hline
$X_3,X_4$ & $0.004$ & $0$ & $0$ & $0$ & $0$ \\ \hline
\end{tabular}
\end{center}
\end{table}

\end{itemize}

From the definition of $X_1$ and $X_2$ we have that $X_2$ is a function of $X_1$ plus an error. We can see from Table \ref{results dim 4-0.01} on page ~\pageref{results dim 4-0.01}, Table \ref{results dim 4-0.1} on page ~\pageref{results dim 4-0.1} and Table \ref{results dim 4-0.25} on page ~\pageref{results dim 4-0.25} that our procedure selects $X_1$ and $X_2$ only between 0.22 and 0.264 proportion of the times, when the other procedures select it in around a 0.5 proportion of the times or more.
In most of the cases our procedure detects that $X_2$ is a
``function" of $X_1$. That is, if $X_1$ is selected to be one of the two explanatory variables, for the other one, the \bf BP\rm, select between $X_3$ and $X_4$.  This way, the \bf BP \rm gains  information about $V_2$ instead of getting redundant information by choosing $X_2$. Note that $\{X_3,X_4\}$ is also a ``good'' choice.

To make the problem a bit more challenging we
enlarge the dimension of the dataset repeating the same variables
plus noisy noninformative errors and also add some i.i.d. noisy
variables. More precisely we consider the following model:

\begin{eqnarray*}
X_j=\left\{
\begin{array}{ccccc}
 V_1+\varepsilon_j, &&\text{ if }  1\leq j \leq 5,
\\ |V_1|+\varepsilon_j, &&\text{ if }  6\leq j \leq 10,
\\ V_2+\varepsilon_j, &&\text{ if }  11\leq j \leq 15,
\\ V_1V_2+\varepsilon_j, &&\text{ if }  16\leq j \leq 20,
\\ \varepsilon_j, &&\text{ if }  21\leq j \leq 23,
\end{array}
\right.
\end{eqnarray*}
where $\varepsilon_j$, $1\leq j\leq 23$ are i.i.d. $N(0,\sigma^2)$, with $\sigma=0.01$ (Case 1), $\sigma=0.1$ (Case 2) and $\sigma=0.25$ (Case 3).

We keep the first and second principal component because on 98\% of the 500 replicates they explain between 70\% and 85\% of the total variance in Case 1, between 70\% and 83\% in Case 2 and between 65\% and 78\% in Case 3. Again we require all the methods to select two
variables. In the three cases, at least 93\% of the times the largest angle is smaller than 25 degrees, hence we consider it is a good decision to look for two variables.

There are five groups of variables $A_1,\ldots,A_5$. Within each
group all the variables only differ on a noisy noninformative
error.
More precisely, the first four groups are $A_j=\{X_k, k=5j-i, i=0,\ldots,4 \}$ $(j=1,\ldots,4)$ and the last one is $A_5=\{X_k, k=21,22,23\}$.
 Clearly $\{\{A_j, A_j\}, \ j=1, \ldots 5\}$, $\{A_1, A_2\}$ and $\{\{A_j, A_5\}, \ j=1, \ldots 4\}$
are ``bad'' choices and also $\{A_2, A_3\}$ and $\{A_2, A_4\}$ are not good choices either.

\begin{table}[!h]
\begin{center}
\caption{Proportion of replicates where each method selects one variable per group of variables.(Case1)}
\label{results dim 23-0.01}
\begin{tabular}{|l|l|l|l|l|l|}
\hline
 & \bf BP\rm & B2  &  B4 & McCabe & SPCA\\ \hline
$A_1,A_1$ & $0.012$ & $0.184$ & $0$ & $0$ & $0$ \\ \hline
$A_1,A_2$ & $0.026$ & $0.202$ & $0.744$ & $0.784$ & $0.744$ \\ \hline
$A_1,A_3$ & $0.214$ & $0.08$ & $0.004$ & $0$ & $0.004$ \\ \hline
$A_1,A_4$ & $0.57$ & $0.114$ & $0.252$ & $0.216$ & $0.252$ \\ \hline
$A_1,A_5$ & $0.028$ & $0.086$ & $0$ & $0$ & $0$ \\ \hline
$A_2,A_2$ & $0$ & $0.076$ & $0$ & $0$ & $0$ \\ \hline
$A_2,A_3$ & $0$ & $0.048$ & $0$ & $0$ & $0$ \\ \hline
$A_2,A_4$ & $0$ & $0.058$ & $0$ & $0$ & $0$ \\ \hline
$A_2,A_5$ & $0$ & $0.032$ & $0$ & $0$ & $0$ \\ \hline
$A_3,A_3$ & $0$ & $0.016$ & $0$ & $0$ & $0$ \\ \hline
$A_3,A_4$ & $0.124$ & $0.02$ & $0$ & $0$ & $0$ \\ \hline
$A_3,A_5$ & $0$ & $0.016$ & $0$ & $0$ & $0$ \\ \hline
$A_4,A_4$ & $0$ & $0.036$ & $0$ & $0$ & $0$ \\ \hline
$A_4,A_5$ & $0$ & $0.026$ & $0$ & $0$ & $0$ \\ \hline
$A_5,A_5$ & $0.026$ & $0.006$ & $0$ & $0$ & $0$ \\ \hline
\end{tabular}
\end{center}
\end{table}

\begin{table}[!h]
\begin{center}
\caption{Proportion of replicates where each method selects one variable per group of variables.(Case2)}
\label{results dim 23-0.1}
\begin{tabular}{|l|l|l|l|l|l|}
\hline
 & \bf BP\rm & B2  &  B4 & McCabe & SPCA\\ \hline
$A_1,A_1$ & $0.006$ & $0.188$ & $0$ & $0$ & $0$ \\ \hline
$A_1,A_2$ & $0.024$ & $0.218$ & $0.744$ & $0.79$ & $0.74$ \\ \hline
$A_1,A_3$ & $0.254$ & $0.07$ & $0.006$ & $0$ & $0.006$ \\ \hline
$A_1,A_4$ & $0.598$ & $0.114$ & $0.25$ & $0.21$ & $0.25$ \\ \hline
$A_1,A_5$ & $0.026$ & $0.076$ & $0$ & $0$ & $0$ \\ \hline
$A_2,A_2$ & $0$ & $0.074$ & $0$ & $0$ & $0$ \\ \hline
$A_2,A_3$ & $0$ & $0.044$ & $0$ & $0$ & $0$ \\ \hline
$A_2,A_4$ & $0$ & $0.072$ & $0$ & $0$ & $0$ \\ \hline
$A_2,A_5$ & $0$ & $0.03$ & $0$ & $0$ & $0$ \\ \hline
$A_3,A_3$ & $0$ & $0.014$ & $0$ & $0$ & $0$ \\ \hline
$A_3,A_4$ & $0.092$ & $0.026$ & $0$ & $0$ & $0$ \\ \hline
$A_3,A_5$ & $0$ & $0.012$ & $0$ & $0$ & $0$ \\ \hline
$A_4,A_4$ & $0$ & $0.034$ & $0$ & $0$ & $0$ \\ \hline
$A_4,A_5$ & $0$ & $0.022$ & $0$ & $0$ & $0$ \\ \hline
$A_5,A_5$ & $0$ & $0.006$ & $0$ & $0$ & $0$ \\ \hline
\end{tabular}
\end{center}
\end{table}

\begin{table}[!h]
\begin{center}
\caption{Proportion of replicates where each method selects one variable per group of variables.(Case3)}
\label{results dim 23-0.25}
\begin{tabular}{|l|l|l|l|l|l|}
\hline
 & \bf BP\rm & B2  &  B4 & McCabe & SPCA\\ \hline
$A_1,A_1$ & $0.028$ & $0.17$ & $0$ & $0$ & $0$ \\ \hline
$A_1,A_2$ & $0.004$ & $0.234$ & $0.736$ & $0.762$ & $0.728$ \\ \hline
$A_1,A_3$ & $0.292$ & $0.094$ & $0.008$ & $0.004$ & $0.008$ \\ \hline
$A_1,A_4$ & $0.604$ & $0.104$ & $0.256$ & $0.234$ & $0.256$ \\ \hline
$A_1,A_5$ & $0.026$ & $0.084$ & $0$ & $0$ & $0$ \\ \hline
$A_2,A_2$ & $0$ & $0.072$ & $0$ & $0$ & $0$ \\ \hline
$A_2,A_3$ & $0$ & $0.032$ & $0$ & $0$ & $0$ \\ \hline
$A_2,A_4$ & $0$ & $0.07$ & $0$ & $0$ & $0$ \\ \hline
$A_2,A_5$ & $0$ & $0.03$ & $0$ & $0$ & $0$ \\ \hline
$A_3,A_3$ & $0$ & $0.022$ & $0$ & $0$ & $0$ \\ \hline
$A_3,A_4$ & $0.046$ & $0.024$ & $0$ & $0$ & $0$ \\ \hline
$A_3,A_5$ & $0$ & $0.012$ & $0$ & $0$ & $0$ \\ \hline
$A_4,A_4$ & $0$ & $0.04$ & $0$ & $0$ & $0$ \\ \hline
$A_4,A_5$ & $0$ & $0.006$ & $0$ & $0$ & $0$ \\ \hline
$A_5,A_5$ & $0$ & $0.006$ & $0$ & $0$ & $0$ \\ \hline
\end{tabular}
\end{center}
\end{table}

Table \ref{results dim 23-0.01} on page ~\pageref{results dim 23-0.01}, Table \ref{results dim 23-0.1} on page ~\pageref{results dim 23-0.1}, and Table \ref{results dim 23-0.25} on page ~\pageref{results dim 23-0.25} exhibit the proportion of replicates where each method selects one variable per group of variables. In the three cases the \bf BP\rm \  achieves the desired results more than 90\%  of the times, while the other
methods fail more than 73\% of the times.

\subsection{Example 2}
\citet{ZHT06} introduced the following simulation
example. They have two ``hidden" factors:
\begin{eqnarray*}
 && V_1\sim N(0,290),
\\&& V_2\sim N(0,300),
\end{eqnarray*}
and a lineal combination of them,
$$ V_3=-0.3V_1+0.925V_2+\varepsilon$$
where $V_1$, $V_2$ and $\varepsilon$ are independent, $\varepsilon \sim N(0,1)$.

Then 10 observable variables are constructed as follows:

\begin{eqnarray*}
X_j=\left\{
\begin{array}{ccc}
 V_1+\varepsilon_j, &&\text{ if }  1\leq j \leq 4,
\\ V_2+\varepsilon_j, &&\text{ if }  5\leq j \leq 8,
\\ V_3+\varepsilon_j, &&\text{ if }  j=9,10,
\end{array}
\right.
\end{eqnarray*}
where $\varepsilon_j$, $1\leq j\leq10$ are i.i.d. $N(0,1)$.

\citet{ZHT06} used the true covariance matrix of
$(X_1,\ldots,X_{10})$ to perform PCA, SPCA and Simple Thresholding.

There are three groups of variables that share the same
information. The first group, which we denote by $A_1$ corresponds to the variables $X_1$ to
$X_4$, the second one, $A_2$, to the variables $X_5$ to $X_8$, and the
third group, $A_3$, to the variables $X_9$ and $X_{10}$. By definition
$V_3$ is also a function of $V_1$ and $V_2$, up to a noisy
noninformative error.

\citet{ZHT06} show that SPCA correctly identifies the sets of
important variables, the first SPCA identifies the group of the variables $A_2$ and the second SPCA the group of the variables
$A_1$. Simple Thresholding incorrectly mixes two variables of $A_2$ with two of $A_3$.

We first consider the population version, using the true
covariance matrix,  where the first two principal components explain $99.6\%$ of the total variance.

This example is outside of the theoretical framework stated in section 2 but this will not be a problem for the implemetation of our procedure.

Clearly it is not adecuate to keep one variable because when the cardinality of $I$ is $1$, $\Sigma_{\mathbf{Y}^{I}}$ has only one positive eigenvalue and then we can not recover two principal components.

%As shown in Table \ref{theoretical results 1 var} on page ~\pageref{theoretical results 1 var} it is clear that it
%is not enough to keep one variable since for the optimal value of
%the theoretical objective function (\ref{functionh1}), the largest angle between the
%original and the blinding components is $80.96$.
%\begin{table}[!h]
%\begin{center}
%\caption{ Largest angle between the original and the blinded components  and theoretical objective
%function's value $h(I)$ when one variable is chosen.}
%\label{theoretical results 1 var}
%\begin{tabular}{|l|l|l|}
%\hline & largest angle in degrees& $h(I)$  \\ \hline
%$A_1$ & $90$ & $1.656$ \\ \hline
%$A_2$  & $80.96$ & $0.871$ \\ \hline
%$A_3$ & $90$ & $1$   \\ \hline
%\end{tabular}
%\end{center}
%\end{table}

If we choose a subset $I$ of cardinal 2, then $\Sigma_{\mathbf{Y}^{I}}$ has two positive eigenvalues of multiplicity one and then, we can recover two principal components. If we select one variable of each of the three groups will be a ``good'' solution. A ``bad'' solution should be to choose the 2 variables within the same group.

\begin{table}[!h]
\begin{center}
\caption{Theoretical objective's value $h(I)$ when  two variables are chosen (one of each group).}
\label{theoretical results 2 var}
\begin{tabular}{|l|l|}
\hline & $h(I)$   \\
\hline $A_1,A_1$ & $1.656$ \\
\hline $A_1,A_2$ & $2.257\times10^{-6}$ \\
\hline $A_1,A_3$ & $2.536\times10^{-5}$ \\
\hline $A_2,A_2$ & $1.028$   \\
\hline $A_2,A_3$ & $0.001$   \\
\hline $A_3,A_3$ & $1$   \\ \hline
\end{tabular}
\end{center}
\end{table}

%\begin{table}[!h]
%\begin{center}
%\caption{Largest angle between the original and the blinded components  and
% theoretical objective's value $h(I)$ when  two variables are chosen (one of each group).}
%\label{theoretical results 2 var}
%\begin{tabular}{|l|l|l|}
%\hline & largest angle in degrees& $h(I)$   \\ \hline $A_1,A_1$ &
%$90$ &
%$1.656$ \\ \hline $A_1,A_2$  & $0.087$ & $2.257\times10^{-6}$ \\
%\hline $A_1,A_3$ & $0.292$ & $2.536\times10^{-5}$   \\ \hline
%$A_2,A_2$ & $90$ & $1.028$   \\ \hline $A_2,A_3$ & $1.891$ &
%$0.001$   \\ \hline $A_3,A_3$ & $90$ & $1$   \\ \hline
%\end{tabular}
%\end{center}
%\end{table}

In Table \ref{theoretical results 2 var} on page ~\pageref{theoretical results 2 var} we can see that $h(I)$ takes large values (larger than one) for the ``bad" choices and small values (smaller than 0.001) for the ``good" choices. Moreover, for the ``good" choices, the largest angle between the original and the blinded components is less than two dergees in all cases. Then, the \textbf{BP} is going to select a
``good" choice. So our method performs perfectly well when the cardinal of $I$ is $2$.

Next, we consider a more realistic situation and perform a small
simulation, where we estimate the covariance matrix generating a
sample of size 50 which we iterate 500 times. We
perform the analysis  considering the two first principal
components, which explain more than the 99\% of the
variance.

%We apply our procedure to select one variable for the first two
%principal components and we obtain angles larger than 20 degrees
%in more than 86\% of the cases.  Then we decided to select
%two variables and in only 4.8\% of the cases the larger angle is
%above 15 degrees and in 2.2\% of the cases larger than 20 degrees.
We apply our procedure to select two variables for the first two principal components and in only 4.8\% of the cases the larger angle was above 15 degrees and in 2.2\% of the cases larger than 20 degrees.

In Table \ref{simulation dim 10} on page ~\pageref{simulation dim 10} we can see that 99.2\%  of the times the \bf BP \rm makes a ``good'' choice. In addition McCabe
and SPCA always make a ``good'' choice, $B4$ on 99\% of the times and $B2$ only makes a ``good'' choice in 76.6\% of the times.

\begin{table}[!h]
\begin{center}
\caption{ Proportion of times where each method selects one
variable per group of variables.}
\label{simulation dim 10}
\begin{tabular}{|l|l|l|l|l|l|}
\hline
 & \bf BP\rm & B2  &  B4 & McCabe & SPCA\\ \hline
$A_1,A_1$ & $0$ & $0.134$ & $0.01$ & $0$ & $0$ \\ \hline
$A_1,A_2$ & $0.426$ & $0.57$ & $0.562$ & $1$ & $0.736$ \\ \hline
$A_1,A_3$ & $0.222$ & $0.108$ & $0.428$ & $0$ & $0.264$ \\ \hline
$A_2,A_2$ & $0$ & $0.096$ & $0$ & $0$ & $0$ \\ \hline
$A_2,A_3$ & $0.344$ & $0.088$ & $0$ & $0$ & $0$ \\ \hline
$A_3,A_3$ & $0.008$ & $0.004$ & $0$ & $0$ & $0$ \\ \hline
\end{tabular}
\end{center}
\end{table}

\section{A real data example}

We consider a dataset obtained from the University of California,
Irvine repository (\citet{UCI}) as an example. This dataset contains the values of
six bio-mechanical characteristic that were used to classify
orthopaedic patients into three classes (normal, disc hernia and
spondylolisthesis groups), and includes data for one hundred
normal patients, sixty patients with a disc hernia, and one
hundred and fifty patients with spondylolisthesis. Each patient is
represented in the dataset in terms of the six bio-mechanical
attributes associated with the shape and orientation of the pelvis
and the lumbar spine, namely (a) pelvic
incidence, (b) pelvic tilt, (c) lumbar lordosis angle, (d) sacral
slope, (e) pelvic radius and (f) grade of spondylolisthesis.

We use our procedure to select one variable for the first two
principal components, which explain 85\% of the variance. We
give the same weight to all components, that is $p_k=1/2$ in
equation (\ref{objetivo}), and the ``grade of spondylolisthesis", variable (f) is selected. In this case the largest angle is $7.5$ degrees,  hence we decide to keep one variable. Cross-validation finds 55 nearest neighbours for variables (a) and
(b), 70 for variable (c), 102 for variable (d) and 39 for variable
(e), while the empirical objective function (\ref{empiricalOBJ}) attains the value 0.017. We use the
Euclidean distance in our procedure, but with the Mahalanobis
distance we get the same results.

\begin{figure}[h!]
\begin{center}
\includegraphics[width=5cm]{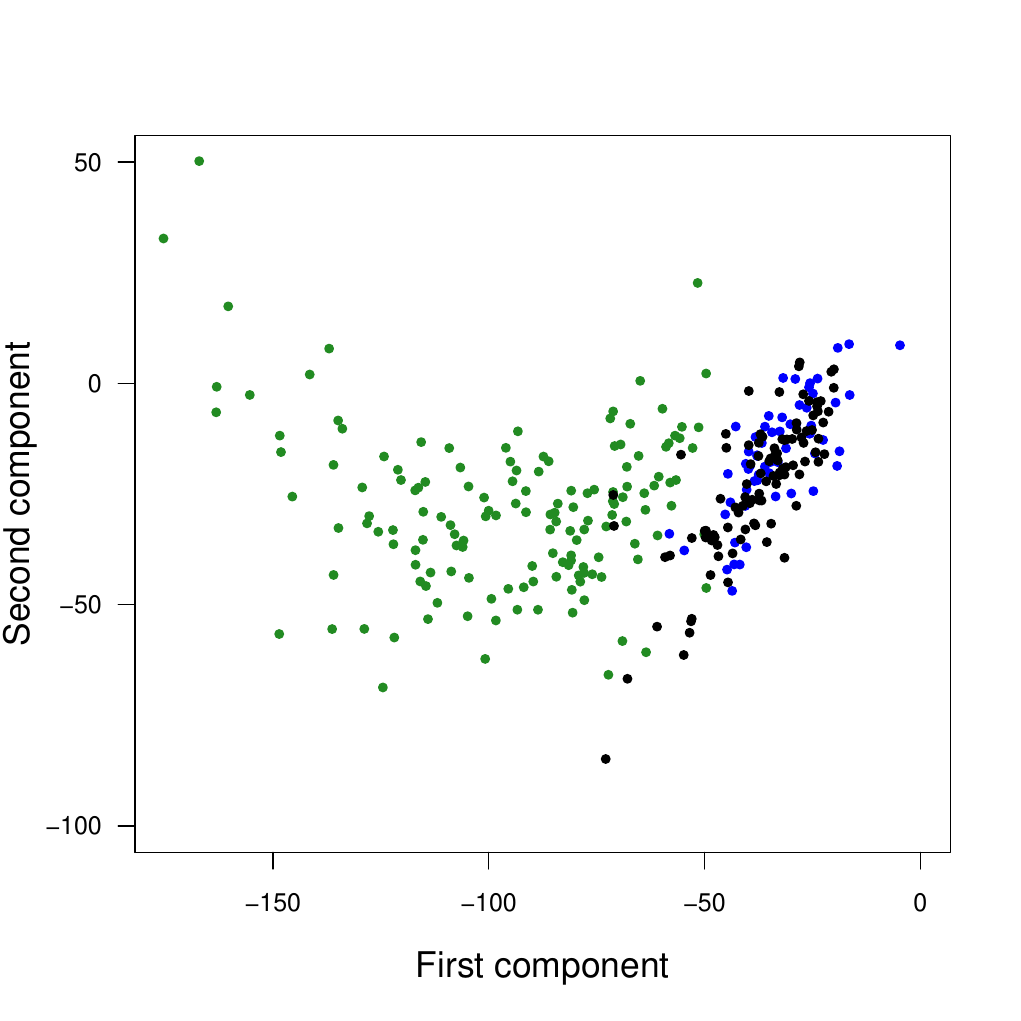}
\includegraphics[width=5cm]{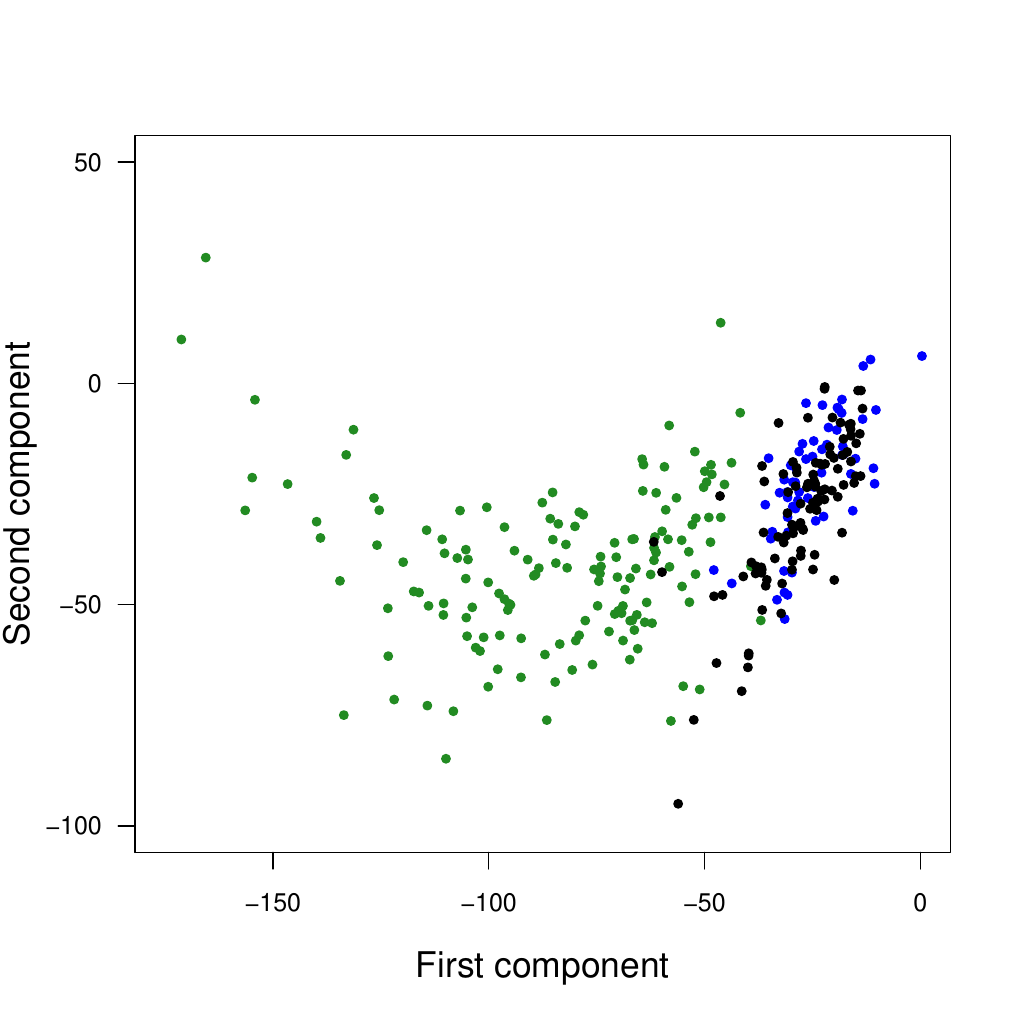}
\end{center}
\caption{Left: Map of the projection of the data on the original
principal components. Right: Map of the projection of the data on
the blinded based principal components. Blue: Disk Hernia
patients. Black: Normal patients. Green: Spondylolisthesis
patients.}\label{alldata 1var}
\end{figure}

The plot on the first two principal components for the example data
looks very similar to the plot on the principal components calculated using
our procedure (Figure \ref{alldata 1var}). It also shows that the
patients with spondylolisthesis are separated from the rest of the
patients, but the normal patients and the patients with disc
hernias are mixed up together.

Next, we calculate the two principal components using the
traditional method and our procedure for the subset of normal and
disc hernia patients, using now two variables. The first two
principal components  explain 75\%  of the variance then we decide
to keep them. We discard to select one variable because in that
case the largest angle is  close to $78$ degrees. For two
variables the largest angle is close to $21$
degrees and for three variables the angle does not decrease very
much ($19$ degrees), so we decide to  keep two variables.

We consider both, the Euclidean and the Mahalanobis distance, and in both cases the variables selected are ``lumbar lordosis angle"
and ``pelvic radius". The value of the objective function (\ref{empiricalOBJ}) with the Euclidean distance is $0.125$,
while with the Mahalanobis distance is $0.141$. Figure
\ref{dhnodata 2var} shows that the plots for the principal
components produced by these procedures look very similar.

\begin{figure}%[h!]
%\begin{center}
\includegraphics[width=3.5cm]{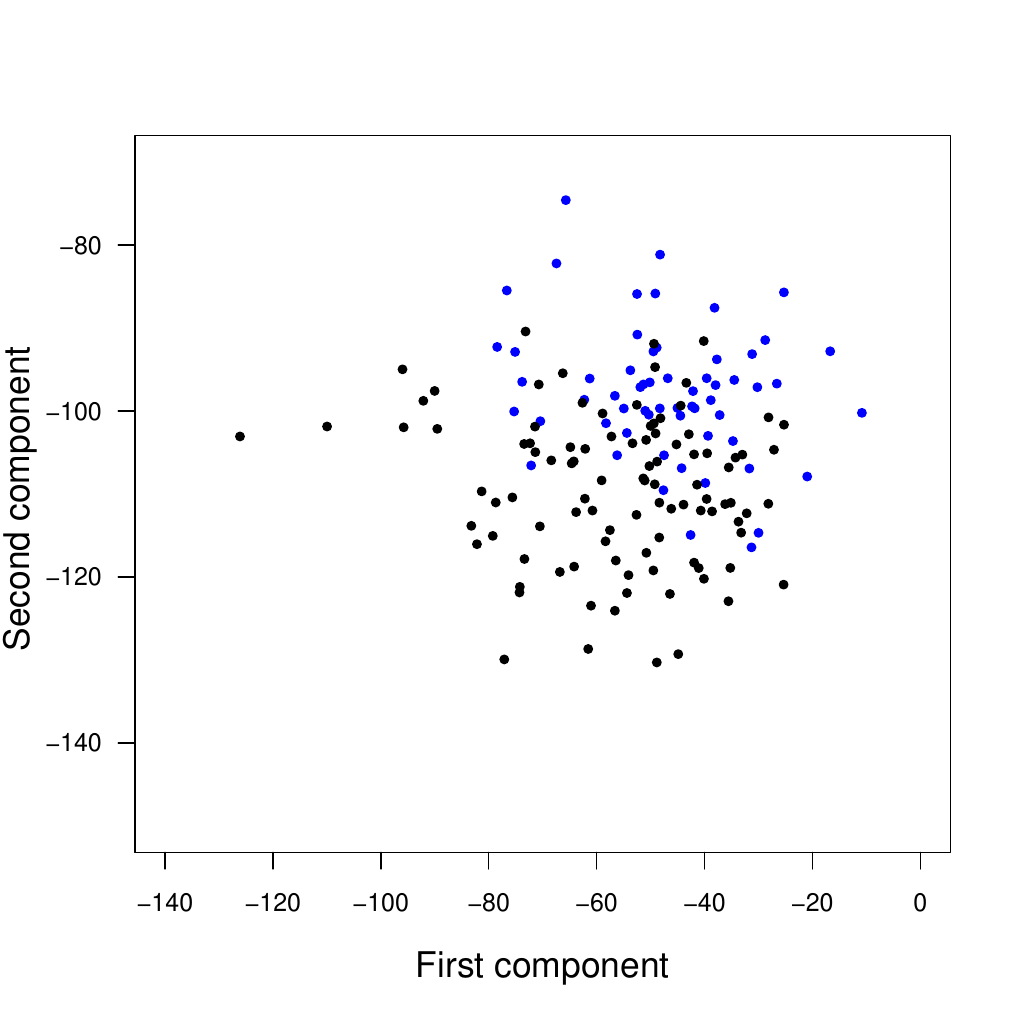}
\includegraphics[width=3.5cm]{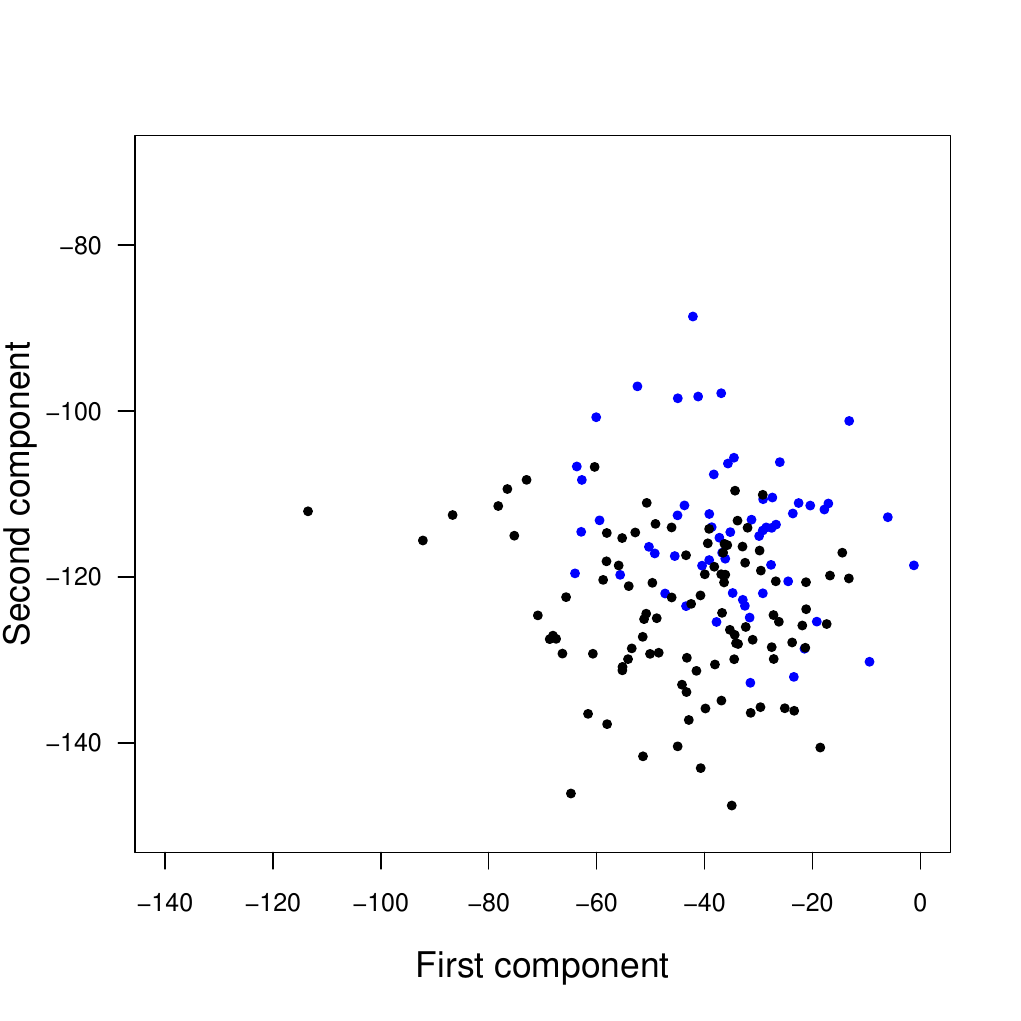}
\includegraphics[width=3.5cm]{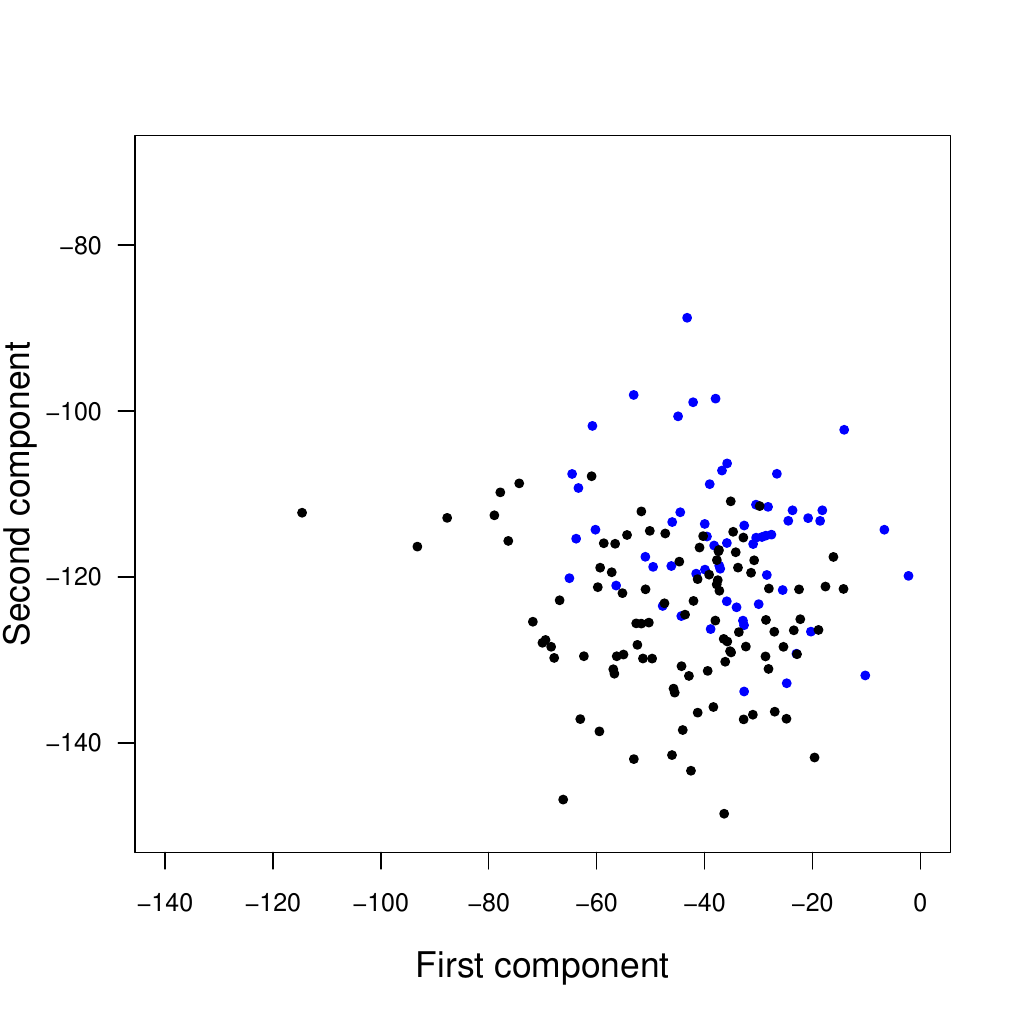}
%\end{center}
\caption{Left: Map of the projection of the data on the original
principal components. Middle: Map of the projection of the data on
the blinded based principal components using Euclidean distance.
Right: Map of the projection of the data on the blinded based
principal components using Mahalanobis distance. Blue: Disk Hernia
patients. Black: Normal patients.}
\label{dhnodata 2var}
\end{figure}

\section{Conclusions}

In this paper we introduce a new variable selection procedure for PCA that aim to gain interpretation in principal components.

On one hand we consider a local approach where we deals with finding a subset of variables that best explains each principal component, and on the other hand we analyzed a global approach when the objective is to find a unique subset to explain the first $q$ principal components at once.

We explained how through the conditional expectation we can choose the variables.

Under wide regular conditions in the conditional expectation estimates and on the covariance matrix, strong consistency results are studied and numerical aspects have also been analysed.

The performance of the procedure has been compared with several well-known variable selection techniques using real and simulated data sets, showing the strengths of the method.

%As we can realize this procedure selects a small subset of the original variables that best explain the principal components.
%
%We gave two approaches. The first one, the local approach, deals with finding a subset of variables that explains a specific principal component. The second one, the global approach, the objective is to find a unique subset to explain the first $q$ principal components at once.
%In addition, the asymptotic behavior of the method is analyzed, showing that under some assumptions the subsets are consistently estimated.
%
%Finally we apply our method to simulated and real examples.
%
%In the simulated examples we compare our method with other classical methods from the literature, as they are, Jolliffe and McCabe variable selection methods and sparce PCA.   The good interpretation power of our procedure can be observed in the examples.

\section*{Appendix: Proof of Theorem \ref{consistencia}.}%\label{app}

First we are going to proof the following Proposition:

\begin{proposition*}
For $k\in\{1,...,p\}$, if
\begin{eqnarray}
\label{convergenciah}
\hat{h}^k_n(I)\rightarrow h^k(I) \text{ a.s. for all } I \in \mathcal{I}_d,
\end{eqnarray}
then it converges uniformly, and
\begin{eqnarray*}
arg \min_{I \in \mathcal{I}_d} \hat{h}^k_n(I) \rightarrow arg \min_{I \in \mathcal{I}_d} h^k(I) \text{ a.s. }
\end{eqnarray*}
\end{proposition*}

\begin{proof}
Since $\mathcal{I}_d$ is a discrete and finite space , the convergence of (\ref{convergenciah}) is uniform.

For the sake of simplicity, let $k=1$ (the proof is analogue for $k\in\{2,...,p\}$).

For all $\delta>0$ there exists $n_0(\omega)$ such that for every $n\geq n_0(\omega)$, with probability one

\begin{eqnarray*}
\left|\hat{h}^1_n(I)-h^1(I)\right|<\delta \text{ for all } I\in\mathcal{I}_d,
\end{eqnarray*}

then,
\begin{eqnarray}
\label{cotahn}
h^1(I)-\delta<\hat{h}^1_n(I)<h^1(I)+\delta \text{ for all } I\in\mathcal{I}_d
\end{eqnarray}

and

\begin{eqnarray}
\label{cotah}
\hat{h}^1_n(I)-\delta<h^1(I)<\hat{h}^1_n(I)+\delta \text{ for all } I\in\mathcal{I}_d.
\end{eqnarray}

From (\ref{henobjetivo1}), we know that exists $\delta_0>0$ such that

\begin{eqnarray*}
\label{hI1}
h^1(I_1)<h^1(I)-\delta_0 \text{ for all } I\notin \mathcal{I}_{1,0}, \text{ for all } I_1 \in \mathcal{I}_{1,0}.
\end{eqnarray*}

By choosing in (\ref{cotahn}) $\delta=\frac{\delta_0}{2}$ we obtain

\begin{eqnarray*}
\hat{h}^1_n(I_1)<h^1(I_1)+\frac{\delta_0}{2}<h^1(I)-\delta_0+\frac{\delta_0}{2}
=h^1(I)-\frac{\delta_0}{2}<\hat{h}^1_n(I)
\end{eqnarray*}

for all $I\notin \mathcal{I}_{1,0}$, for all $I_1 \in \mathcal{I}_{1,0}$, if $n\geq n_0(\omega)$ with  probability one.

That means, that there exists $n_0(\omega)$ such that, if $n\geq n_0(\omega)$, with probability one, if we choose $ I_1 \in \mathcal{I}_{1,0}$, then $I_1$ minimizes $\hat{h}^1_n(I)$.

Moreover, from (\ref{empirical}), we have that for each $n$ fix, exists $\delta_1>0$, such that

\begin{eqnarray*}
\hat{h}^1_n(I^1)<\hat{h}^1_n(I)-\delta_1 \text{ for all } I\notin \hat I_{1n}, \text{ for all } I^1 \in \hat I_{1n}.
\end{eqnarray*}

Let $\delta=\frac{\delta_1}{2}$ in (\ref{cotah}) and $n\geq n_1(\omega)$,

\begin{eqnarray*}
h^1(I^1)<\hat{h}^1_n(I^1)+\frac{\delta_1}{2}<\hat{h}^1_n(I)-\delta_1+\frac{\delta_1}{2}
=\hat{h}^1_n(I)-\frac{\delta_1}{2}<h^1(I)
\end{eqnarray*}

for all $I\notin \hat I_{1n}$, and for all $I^1 \in \hat I_{1n}$ with probability one.

That means, that exists $n_1(\omega)$ such that, if $n\geq n_1(\omega)$ for all $ I^1 \in \hat I_{1n}$, $I^1$ minimizes $h^1(I)$ with probability one.

In conclusion, if $n>max\{n_0,n_1\}$, for all $ I_1 \in \mathcal{I}_{1,0}$, $I_1$ minimizes $\hat{h}^1_n(I)$ and  for all $ I^1 \in \hat I_{1n}$, $I^1$ minimizes $h^1(I)$ with probability one.

As the proof is analogue for $k\in\{2,...,p\}$,
\begin{equation*}
\label{objetivo1} \mathcal{I}_{k,0} = argmin_{I \in \mathcal{I}_d} h^k(I),
\end{equation*}
and
$$\hat I_{kn} = argmin_{I \in \mathcal{I}_d} \hat h^k_n(I)$$
this implies,
\begin{eqnarray*}
arg \min_{I \in \mathcal{I}_d} \hat{h}^k_n(I) \rightarrow arg \min_{I \in \mathcal{I}_d} h^k(I) \text{ a.s. for } k\in\{1,...,p\}.
\end{eqnarray*}
\end{proof}

\begin{proof}\textbf{(of Theorem)}

To prove the Theorem \ref{consistencia} then it is enough to see that for each $I$ the empirical objective function (\ref{empirical}) converges a.s. to the theoretical objective function (\ref{functionhk}). To prove it, let us see that
\begin{eqnarray*}
\left\|\hat{\alpha}_n^k(\mathbb{P}_n)-\hat{\alpha}_n^k(\mathbb{P}_{n,\textbf{Y}^I})\right\|\rightarrow\left\|\alpha^k(\mathbb{P})-\alpha^k(\mathbb{P}_{\textbf{Y}^I})\right\| \ \text{ a.s. for all } k.
\end{eqnarray*}

\citet{daux} proved that under \textbf{H1}, if
\begin{equation*}
\label{convergencia covarianza}
\left\Vert\hat{\Sigma}_n-\Sigma\right\Vert_{\infty}=
\underset{\left\Vert u \right\Vert = 1}{max}\left\Vert
\left(\hat{\Sigma}_n-\Sigma \right) \left(u\right)\right\Vert
\rightarrow 0 \ a.s.,
\end{equation*}
then
\begin{eqnarray*}
\left\|\hat{\alpha}^k_n(\mathbb{P}_n) - \alpha^k(\mathbb{P})\right\| \rightarrow 0 \ \text{ a.s. for all } 1\leq k\leq p,
\end{eqnarray*}
where $\hat{\alpha}^k_n(\mathbb{P}_n)$ (respec. $\alpha^k(\mathbb{P})$)
are the eigenvectors of $\hat{\Sigma}_n$, that denotes the empirical covariance matrix associated with $\mathbb{P}_n$ (respec. $\Sigma$ denotes the covariance matrix associated with $\mathbb{P}$).

So, if we prove that
\begin{equation*}
\label{convergencia} \underset{\left\Vert u \right\Vert =
1}{max}\left\Vert \left(\hat{\Sigma}_n(I)-\Sigma(I) \right)
\left(u\right)\right\Vert \rightarrow 0 \ a.s.
\end{equation*}
then
\begin{equation*}
\Vert \hat{\alpha}^k_n(\mathbb{P}_{n,\mathbf{Y}^{I}}) -
\alpha^k(\mathbb{P}_{\mathbf{Y}^{I}})\Vert \to 0 \ \ a.s.  \ \ \forall
1\leq k\leq p ,
\end{equation*}
where $\hat{\alpha}^k_n(\mathbb{P}_{n,\mathbf{Y}^{I}})$ (respec. $\alpha^k(\mathbb{P}_{\mathbf{Y}^{I}})$)
are the eigenvectors of $\hat{\Sigma}_n(I)$, that denotes the empirical covariance matrix associated with $\mathbb{P}_{n,\mathbf{Y}^{I}}$ (respec. $\Sigma(I)$ denotes the covariance matrix associated with $\mathbb{P}_{\mathbf{Y}^{I}}$).

Note, that $I=\{1,\ldots,p\}$ is the classic case. We will show that this also holds for any $I\subseteq \{1, \ldots,
p\}$, that is

\begin{equation*}
\label{convergencia} \underset{\left\Vert u \right\Vert =
1}{max}\left\Vert \left(\hat{\Sigma}_n(I)-\Sigma(I) \right)
\left(u\right)\right\Vert \rightarrow 0 \ a.s.
\end{equation*}

To simplify the notation, we can assume (without losing generality)
that $I=\{1,\ldots, d\}$, then
\begin{equation}
\nonumber
\mathbf{Y}_j= \left(X_j[1],\ldots,X_j[d],g^1_n(\mathbf{X}_j[I]),
\ldots,g^{p-d}_n(\mathbf{X}_j[I])\right)^t,
\end{equation}
where
$g^i_n(z)$ is a uniformly consistent nonparametric estimate of $g^i(z)$. Specifically, $g^i_n$ will fulfil the following assumption,
\begin{eqnarray*}
g_n^{l}(\mathbf{X}_j[I]) \rightarrow g^{l}(\mathbf{X}_j[I]) \ a.s.\text{ for any $l$, for
all $j$, uniformly}.
\end{eqnarray*}
We define a non observable auxiliary vector,
\begin{equation}
\nonumber \mathbf{Z}_j=\left(X_j[1], \ldots,X_j[d],g^1(\mathbf{X}_j[I]),
\ldots, g^{p-d}(\mathbf{X}_j[I])\right)^t,
\end{equation}
where $g^i(z)=E(X[d+i]|\mathbf X[I]=z)$.

The proof will be complete if we show that

\begin{equation}\label{convergencia1}
\underset{\left\Vert u \right\Vert = 1}{max}\left\Vert
\left(\Sigma_{\mathbf{Z}}-\Sigma(I) \right)
\left(u\right)\right\Vert \rightarrow 0 \ a.s.,
\end{equation}

and that

\begin{equation}\label{convergencia2}
\underset{\left\Vert u \right\Vert = 1}{max}\left\Vert
\left(\Sigma_{\mathbf{Y}} - \Sigma_{\mathbf{Z}} \right)
\left(u\right)\right\Vert \rightarrow 0 \ a.s.,
\end{equation}

where $\Sigma_{\mathbf{Y}}=\hat{\Sigma}_n(I)$.\\

Considering,
{\footnotesize{
$$\overline{X[i]}=\frac{1}{n}\sum_{j=1}^n X_{j}[i],\quad
\overline{g_n^{l}(\mathbf X[I])}=\frac{1}{n}\sum_{j=1}^n g_n^{l}(\mathbf{X}_j[I]),\quad
\overline{g^{l}(\mathbf X[I])}=\frac{1}{n}\sum_{j=1}^n g^{l}(\mathbf{X}_j[I])
$$}}
the three matrices are
{\footnotesize{
\begin{eqnarray*}
 (\Sigma(I))_{ii'}= \left\{\begin{array}{cccc} cov(X[i],X[i'])
\text{ if} \ 1\leq i,i'\leq d,
\\ cov(X[i],\mathbb{E}(X[i']|\mathbf X[I])) \text{ if } 1\leq i\leq d < i' \leq p,
\\ cov(\mathbb{E}(X[i]|\mathbf X[I]),X[i']) \text{ if }  1 \leq i' \leq d < i \leq p,
\\ cov(\mathbb{E}(X[i]|\mathbf X[I]),\mathbb{E}(X[i']|\mathbf X[I])) \text{ if }  d+1 \leq i,i' \leq
p,
\end{array}
\right.
\end{eqnarray*}}}

 {\footnotesize{
\begin{eqnarray*}
(\Sigma_{\mathbf{Y}})_{ii'}= \left\{\begin{array}{cccc}
\frac{1}{n}\sum_{j=1}^n(X_{j}[i]-\overline{X[i]})(X_{j}[i']-\overline{X[i']})
\text{ if } 1 \leq i,i'\leq d,
\\ \frac{1}{n}\sum_{j=1}^n(X_{j}[i]-\overline{X[i]})(g_n^{i'-d}(\mathbf{X}_j[I])-\overline{g_n^{i'-d}(\mathbf X[I])})  \text{ if } 1\leq i\leq d < i' \leq p,
\\ \frac{1}{n}\sum_{j=1}^n(g_n^{i-d}(\mathbf{X}_j[I])-\overline{g_n^{i-d}(\mathbf X[I])})(X_{j}[i']-\overline{X[i']}) \text{ if }  1 \leq i' \leq
d < i \leq p,
\\ \frac{1}{n}\sum_{j=1}^n(g_n^{i-d}(\mathbf{X}_j[I])-\overline{g_n^{i-d}(\mathbf X[I])})
(g_n^{i'-d}(\mathbf{X}_j[I])-\overline{g_n^{i'-d}(\mathbf X[I])}) \text{ if
}  d+1 \leq i,i' \leq p,
\end{array}
\right.
\end{eqnarray*}}}

{\footnotesize{
\begin{eqnarray*}
(\Sigma_{\mathbf{Z}})_{ii'}= \left\{\begin{array}{cccc}
\frac{1}{n}\sum_{j=1}^n(X_{j}[i]-\overline{X[i]})(X_{j}[i']-\overline{X[i']})
\text{ if } 1 \leq i,i'\leq d,
\\ \frac{1}{n}\sum_{j=1}^n(X_{j}[i]-\overline{X[i]})(g^{i'-d}(\mathbf{X}_j[I])-\overline{g^{i'-d}(\mathbf X[I])})  \text{ if } 1\leq i\leq d < i' \leq p,
\\ \frac{1}{n}\sum_{j=1}^n(g^{i-d}(\mathbf{X}_j[I])-\overline{g^{i-d}(\mathbf X[I])})(X_{j}[i']-\overline{X[i']}) \text{ if }  1 \leq i' \leq
d < i \leq p,
\\ \frac{1}{n}\sum_{j=1}^n(g^{i-d}(\mathbf{X}_j[I])-\overline{g^{i-d}(\mathbf X[I])})
(g^{i'-d}(\mathbf{X}_j[I])-\overline{g^{i'-d}(\mathbf X[I])}) \text{ if }
d+1 \leq i,i' \leq p.
\end{array}
\right.
\end{eqnarray*}}}

We will now prove the convergence in (\ref{convergencia1}) and
(\ref{convergencia2}).

First, we show that
\begin{equation*}
\underset{\left\Vert u \right\Vert = 1}{max}\left\Vert
\left(\Sigma_{\mathbf{Z}}-\Sigma(I) \right) \left(u\right)\right\Vert
\rightarrow 0 \ a.s.
\end{equation*}

It is sufficient to prove that each of the coordinates of the matrix
$\Sigma_{\mathbf{Z}}-\Sigma(I)$ converge to zero.

Set $1\leq i,i'\leq p$, then $(\Sigma_{\mathbf{Z}})_{ii'} \rightarrow (\Sigma(I))_{ii'} \ a.s.$ and then
$(\Sigma_{\mathbf{Z}}-\Sigma(I))_{ii'}\rightarrow 0 \ a.s.$ $\forall i,i'=1,...,p$.

Because we are using a finite dimensional space, and each of the
coordinates converge to zero, it holds that

\begin{equation*}
\underset{\left\Vert u \right\Vert = 1}{max}\left\Vert
\left(\Sigma_{\mathbf{Z}}-\Sigma(I) \right) \left(u\right)\right\Vert
\rightarrow 0 \ a.s.
\end{equation*}

Now, we show that

\begin{equation*}
\underset{\left\Vert u \right\Vert = 1}{max}\left\Vert
\left(\Sigma_{\mathbf{Y}} - \Sigma_{\mathbf{Z}} \right)
\left(u\right)\right\Vert \rightarrow 0 \ a.s.
\end{equation*}

As before, we are going to prove that each of the coordinates of the matrix
$\Sigma_{\mathbf{Y}}-\Sigma_{\mathbf{Z}}$ converge to zero.\\

For better understanding, we define
$$\cG(l,j,n)=g_n^{l}(\mathbf{X}_j[I])-\overline{g_n^{l}(\mathbf
X[I])}-g^{l}(\mathbf{X}_j[I])+\overline{g^{l}(\mathbf X[I])}.$$

$\bullet$ If $1\leq i,i'\leq d$, $(\Sigma_{\mathbf{Y}})_{ii'}=(\Sigma_{\mathbf{Z}})_{ii'}$ then $(\Sigma_{\mathbf{Y}} - \Sigma_{\mathbf{Z}})_{ii'}=0$.\\

$\bullet$ If $1\leq i\leq d, d+1 \leq i' \leq p$,

\begin{eqnarray*}
(\Sigma_{\mathbf{Y}} - \Sigma_{\mathbf{Z}})_{ii'}
&=&\frac{1}{n}\sum_{j=1}^n(X_{j}[i]-\overline{X[i]})(g_n^{i'-d}(\mathbf{X}_j[I])-\overline{g_n^{i'-d}(\mathbf
X[I])}-g^{i'-d}(\mathbf{X}_j[I])+\overline{g^{i'-d}(\mathbf X[I])})\\
&=&\frac{1}{n}\sum_{j=1}^n(X_{j}[i]-\overline{X[i]})\cG(i'-d,j,n).
\end{eqnarray*}

From the Cauchy-Schwarz inequality, we have that

\begin{eqnarray*}
\left|\frac{1}{n}\sum_{j=1}^n(X_{j}[i]-\overline{X[i]})\cG(i'-d,j,n)\right|\leq
\left( \frac{1}{n}\sum_{j=1}^n(X_{j}[i]-\overline{X[i]})^2
\right)^ \frac{1}{2} \left(\frac{1}{n}\sum_{j=1}^n(\cG(i'-d,j,n))^2\right)^\frac{1}{2}
\end{eqnarray*}

On the other hand,
\begin{eqnarray}
\label{conv_var}
\frac{1}{n}\sum_{j=1}^n(X_{j}[i]-\overline{X[i]})^2 \rightarrow
Var(X[i]) \ a.s.
\end{eqnarray}

In addition, since
\begin{eqnarray*}
g_n^{l}(\mathbf{X}_j[I]) \rightarrow g^{l}(\mathbf{X}_j[I]) \ a.s. \text{ uniformly, for any $l$,
for all $j$},
\end{eqnarray*}
then for any l,
\begin{eqnarray*}
\overline{g_n^{l}(\mathbf{X}[I])}-\overline{g^{l}(\mathbf{X}[I])}=\frac{1}{n}\sum_{j=1}^n g_n^{l}(\mathbf{X}_j[I])-g^{l}(\mathbf{X}_j[I])\leq \max_{1\leq j\leq n} \{g_n^{l}(\mathbf{X}_j[I])-g^{l}(\mathbf{X}_j[I])\} \rightarrow 0 \ a.s.
\end{eqnarray*}

and therefore,

\begin{eqnarray*}
\cG(l,j,n)\rightarrow 0 \ a.s. \text{ for any $l$,
for all $j$}.
\end{eqnarray*}

This implies

\begin{eqnarray}
\label{conv_g_particular}
\frac{1}{n}\sum_{j=1}^n (\cG(l,j,n))^2\rightarrow 0 \ a.s. \text{ for any $l$}.
\end{eqnarray}

Thus, by (\ref{conv_var}) and  (\ref{conv_g_particular}) we have

\begin{eqnarray}
\label{ultima_convergencia}
\left|\frac{1}{n}\sum_{j=1}^n(X_{j}[i]-\overline{X[i]})\cG(i'-d,j,n)\right| \rightarrow 0 \ a.s.
\end{eqnarray}

which is what we required for our proof.\\

$\bullet$ If $d+1 \leq i \leq p ,1 \leq i' \leq d$, from the simetry of the matrices and (\ref{ultima_convergencia}), we have that $(\Sigma_{\mathbf{Y}} - \Sigma_{\mathbf{Z}})_{ii'}\rightarrow 0 \ a.s.$\\

$\bullet$ If $d+1 \leq i \leq p , d+1 \leq i' \leq p$,

\begin{eqnarray*}
(\Sigma_{\mathbf{Y}} - \Sigma_{\mathbf{Z}})_{ii'}&=&
\frac{1}{n}\sum_{j=1}^n\left[(g_n^{i-d}(\mathbf{X}_j[I])-\overline{g_n^{i-d}(\mathbf X[I])})(g_n^{i'-d}(\mathbf{X}_j[I])-\overline{g_n^{i'-d}(\mathbf{X}[I])})\right]-\\
&& \left[(g^{i-d}(\mathbf{X}_j[I])-\overline{g^{i-d}(\mathbf X[I])})(g^{i'-d}(\mathbf{X}_j[I])-\overline{g^{i'-d}(\mathbf
X[I])})\right].
\end{eqnarray*}

If we add and subtract $(g^{i-d}(\mathbf{X}_j[I])-\overline{g^{i-d}(\mathbf X[I])})(g_n^{i'-d}(\mathbf{X}_j[I])-\overline{g_n^{i'-d}(\mathbf{X}[I])})$, and rearrange, we get that $(\Sigma_{\mathbf{Y}} - \Sigma_{\mathbf{Z}})_{ii'}$ is the sum of
\begin{eqnarray}
\label{primerasuma}
\frac{1}{n}\sum_{j=1}^n (g_n^{i'-d}(\mathbf{X}_j[I])-\overline{g_n^{i'-d}(\mathbf{X}[I])})\cG(i-d,j,n)
\end{eqnarray}
and
\begin{eqnarray*}
\frac{1}{n}\sum_{j=1}^n (g^{i-d}(\mathbf{X}_j[I])-\overline{g^{i-d}(\mathbf X[I])}) \cG(i'-d,j,n)
\end{eqnarray*}

were (\ref{primerasuma}) can be rewrited as
\begin{eqnarray*}
\frac{1}{n}\sum_{j=1}^n [\cG(i'-d,j,n)+(g^{i'-d}(\mathbf{X}_j[I])-\overline{g^{i'-d}(\mathbf X[I])})]\cG(i-d,j,n)
\end{eqnarray*}

or

\begin{eqnarray*}
\frac{1}{n}\sum_{j=1}^n \cG(i'-d,j,n)\cG(i-d,j,n)
+
\frac{1}{n}\sum_{j=1}^n (g^{i'-d}(\mathbf{X}_j[I])-\overline{g^{i'-d}(\mathbf X[I])})\cG(i-d,j,n).
\end{eqnarray*}

Then we have that
\begin{eqnarray*}
(\Sigma_{\mathbf{Y}} - \Sigma_{\mathbf{Z}})_{ii'}&=&
\frac{1}{n}\sum_{j=1}^n \cG(i'-d,j,n)\cG(i-d,j,n)\\
&+&
\frac{1}{n}\sum_{j=1}^n (g^{i'-d}(\mathbf{X}_j[I])-\overline{g^{i'-d}(\mathbf X[I])})\cG(i-d,j,n)\\
&+&
\frac{1}{n}\sum_{j=1}^n (g^{i-d}(\mathbf{X}_j[I])-\overline{g^{i-d}(\mathbf X[I])}) \cG(i'-d,j,n).
\end{eqnarray*}

From triangular and Cauchy-Schwarz inequality,
\begin{eqnarray*}
\left|(\Sigma_{\mathbf{Y}} - \Sigma_{\mathbf{Z}})_{ii'}\right|
&\leq&
\left(\frac{1}{n}\sum_{j=1}^n (\cG(i'-d,j,n))^2\right)^{1/2}
\left(\frac{1}{n}\sum_{j=1}^n (\cG(i-d,j,n))^2\right)^{1/2}\\
&+&
\left(\frac{1}{n}\sum_{j=1}^n (g^{i'-d}(\mathbf{X}_j[I])-\overline{g^{i'-d}(\mathbf X[I])})^2\right)^{1/2}
\left(\frac{1}{n}\sum_{j=1}^n (\cG(i-d,j,n))^2\right)^{1/2}\\
&+&
\left(\frac{1}{n}\sum_{j=1}^n (g^{i-d}(\mathbf{X}_j[I])-\overline{g^{i-d}(\mathbf X[I])})^2\right)^{1/2}
\left(\frac{1}{n}\sum_{j=1}^n (\cG(i'-d,j,n))^2\right)^{1/2}
\end{eqnarray*}

On the other hand, we have that
\begin{eqnarray*}
\frac{1}{n}\sum_{j=1}^n(g^{l}(\mathbf{X}_j[I])-\overline{g^{l}(\mathbf X[I])})^2 \rightarrow Var(g^{l}(\mathbf{X}[I])) \ a.s. \text{ for any $l$}.
\end{eqnarray*}

If $Var(g^{l}(\mathbf{X}[I]))<\infty$ for any $l$ and $I$,
(\ref{conv_g_particular}) entails $\left|(\Sigma_{\mathbf{Y}} - \Sigma_{\mathbf{Z}})_{ii'}\right|\to 0$.\\

As $Var(X[l])<\infty$, then  for any $l$ and any $I$,

$$Var(X[l])=E(Var(X[l]|\mathbf{X}[I]))+Var(E(X[l]|\mathbf{X}[I]))=E(Var(X[l]|\mathbf{X}[I]))+Var(g^{l}(\mathbf{X}[I]))$$

We now that $E(Var(X[l]|\mathbf{X}[I]))>0$, so we conclude that $Var(g^{l}(\mathbf{X}[I]))<\infty$

\bigskip

We have proved that all the coordinates of the matrix
$\Sigma_{\mathbf{Y}}-\Sigma_{\mathbf{Z}}$ converge to zero, so
\begin{equation*}
\underset{\left\Vert u \right\Vert = 1}{sup}\left\Vert
\left(\Sigma_{\mathbf{Y}} - \Sigma_{\mathbf{Z}} \right)
\left(u\right)\right\Vert \rightarrow 0 \ a.s.
\end{equation*}

Now we are ready to complete the proof of Theorem \ref{consistencia}. From (\ref{convergencia1}) and (\ref{convergencia2}) we are able to use the result from \citet{daux} to derive that,

\begin{equation}
\label{convh}
\hat{h}^k_n(I) = \Vert \hat{\alpha}^k_n(\mathbb{P}_n) -
\hat{\alpha}^k_n(\mathbb{P}_{n,\mathbf{Y}^{I}})\Vert^2 \to \Vert \alpha^{k}(\mathbb{P})
- \alpha^k(\mathbb{P}_{\mathbf{Y}^{I}})\Vert^2 = h^k(I) \ \ a.s.
\end{equation}

Indeed, we have that

\begin{eqnarray*}
&&\Vert \hat{\alpha}^k_n(\mathbb{P}_n) -
\hat{\alpha}^k_n(\mathbb{P}_{n,\mathbf{Y}^{I}})\Vert - \Vert \alpha^k(\mathbb{P})
- \alpha^k(\mathbb{P}_{\mathbf{Y}^{I}})\Vert  \leq
\\ &&\Vert
\hat{\alpha}^k_n(\mathbb{P}_n) - \alpha^k(\mathbb{P})\Vert + \Vert \alpha^k(\mathbb{P})
- \alpha^k(\mathbb{P}_{\mathbf{Y}^{I}})\Vert +
\\ &&\Vert
\alpha^k(\mathbb{P}_{\mathbf{Y}^{I}}) -
\hat{\alpha}^k_n(\mathbb{P}_{n,\mathbf{Y}^{I}})\Vert-\Vert \alpha^{k}(\mathbb{P}) -
\alpha^k(\mathbb{P}_{\mathbf{Y}^{I}})\Vert =
\\ &&\Vert
\hat{\alpha}^k_n(\mathbb{P}_n) - \alpha^{k}(\mathbb{P})\Vert+\Vert
\alpha^{k}(\mathbb{P}_{\mathbf{Y}^{I}}) -
\hat{\alpha}^k_n(\mathbb{P}_{n,\mathbf{Y}^{I}})\Vert \to 0 \ a.s.
\end{eqnarray*}

and

\begin{eqnarray*}\nonumber
&&\Vert \alpha^{k}(\mathbb{P}) - \alpha^k(\mathbb{P}_{\mathbf{Y}^{I}})\Vert -\Vert
\hat{\alpha}^k_n(\mathbb{P}_n) -
\hat{\alpha}^k_n(\mathbb{P}_{n,\mathbf{Y}^{I}})\Vert \leq
\\&&\nonumber \Vert \alpha^{k}(\mathbb{P}) - \hat{\alpha}^k_n(\mathbb{P}_n)\Vert +
\Vert \hat{\alpha}^k_n(\mathbb{P}_n) -
\hat{\alpha}^k_n(\mathbb{P}_{n,\mathbf{Y}^{I}})\Vert +
\\&&\nonumber  \Vert
\hat{\alpha}^k_n(\mathbb{P}_{n,\mathbf{Y}^{I}}) -
\alpha^{k}(\mathbb{P}_{\mathbf{Y}^{I}})\Vert-\Vert \hat{\alpha}^k_n(\mathbb{P}_n)
- \hat{\alpha}^k_n(\mathbb{P}_{n,\mathbf{Y}^{I}})\Vert =
\\  &&\Vert
\alpha^{k}(\mathbb{P})-\hat{\alpha}^k_n(\mathbb{P}_n)\Vert+\Vert
\hat{\alpha}^k_n(\mathbb{P}_{n,\mathbf{Y}^{I}}) -
\alpha^{k}(\mathbb{P}_{\mathbf{Y}^{I}})\Vert \to 0 \ a.s.
\end{eqnarray*}
which entails (\ref{convh}).

\

Finally, (\ref{convh}) and the Proposition implies
\begin{eqnarray*}
arg \min_{I \in \mathcal{I}_d} \hat{h}^k_n(I) \rightarrow arg \min_{I \in \mathcal{I}_d} h^k(I) \text{ a.s. }
\end{eqnarray*}

which concludes the proof.

\end{proof}

\

\section*{Acknowledgements}

This work has been partially supported  by Grant  pict 2008-0921
from ANPCyT (Agencia Nacional de Promoci\'on Cient\'{i}fica y Tecnol\'ogica), Argentina. We are very grateful to Dr. Ricardo Fraiman for his invaluable assistance and Dra. Marcela Svarc for her helpful suggestions.
The authors would like to thank the referees for their constructive comments which improve significantly this work.

\

\end{document}